\documentclass[12pt,a4paper]{amsart}
\usepackage{amsfonts, amsmath, amssymb, amscd, amsthm, graphicx, xspace,url}
\usepackage{hyperref}
\usepackage[all]{xy}
\usepackage[usenames, dvipsnames]{color}
\usepackage{cleveref}
\usepackage[normalem]{ulem}
\usepackage{cancel}
\usepackage{xcolor}
\newtheorem{thm}{Theorem}[section]
\newtheorem{cor}[thm]{Corollary}
\newtheorem{lemma}[thm]{Lemma}
\newtheorem{prop}[thm]{Proposition}

\theoremstyle{definition}
\newtheorem{deff}[thm]{Definition}

\theoremstyle{remark}
\newtheorem{rem}[thm]{Remark}

\newcommand{\Hb}{\mathbb{H}}
\newcommand{\Z}{\mathbb{Z}}
\newcommand{\N}{\mathbb{N}}

\newcommand{\Q}{{\mathbb Q}}
\newcommand{\R}{{\mathbb R}}
\newcommand{\C}{{\mathbb C}}

\newcommand{\HQ}{{\mathbb H}}
\newcommand{\GL}{{\rm GL}}
\newcommand{\PGL}{{\rm PGL}}
\newcommand{\SL}{{\rm SL}}

\newcommand{\SO}{{\rm SO}}

\newcommand{\SG}{{\rm SG}}
\newcommand{\PSL}{{\rm PSL}}

\newcommand{\matriz}[1]{\begin{array} #1 \end{array}}

\newcommand{\GEN}[1]{\left\langle #1 \right\rangle}

\newcommand{\quat}[2]{\left( \frac{#1}{#2} \right)}

\newcommand{\U}{\mathcal{U}}

\newenvironment{proofof}{\par\noindent \textit{Proof of }}{\qed\par\bigskip}

\begin{document}
\title{ Aritmethic lattices of $\SO(1,n)$  and units of  group rings}

\author{Sheila Chagas}
\address[Sheila Chagas]{Universidade de
	Bras\'{i}lia, Departamento de Matem\'atica, 70910-900 Bras\'{\i}lia-DF, Brazil.}
\email{sheila@mat.unb.br}

\author{\'{A}ngel del Rio}
\thanks{the second author was partially supported by Grant PID2020-113206GB-I00 funded by MCIN/AEI/10.13039/501100011033 and by Fundación Séneca (22004/PI/22)}

\address[\'{A}ngel del R\'{i}o]{Departamento de Matem\'{a}ticas, Universidad de Murcia, 30100, Murcia, Spain.} 
\email{adelrio@um.es}

\author{Pavel A. Zalesskii}
\thanks{the third author was partially supported by CNPq}
\address[Pavel Zalesskii]{Universidade de
	Bras\'{i}lia, Departamento de Matem\'atica, 70910-900 Bras\'{\i}lia-DF, Brazil.}
\email{pz@mat.unb.br}

\date{\footnotesize\today}



\begin{abstract}
 We establish that standard arithmetic subgroups of a special orthogonal group $\SO(1,n)$ are conjugacy separable. As an application we deduce this property for unit groups of certain integer group rings.
 We also prove that finite quotients of group of units of any of these group rings  determines the original group ring.
\end{abstract}
\maketitle

\section{Introduction}

In recent years there has been a great deal of interest in detecting properties of a group $G$  via its finite quotients, or more conceptually by its profinite completion $\widehat{G}$. 
To detect  properties of combinatorial  or geometric nature in the profinite completion worked  particularly well for so called virtually compact special groups \cite{WiltonZalesskii2016,WiltonZalesskii2019}. These groups defined and studied  by Wise and his collaborators are playing a central role in the modern geometric group theory (see \cite{BergeronHaglundWise2011,  H-W-1, Wise_sm-c, Wise-qc-h} and many others). The important property of these groups is that they admit a  hierarchy that survives in the profinite completion and so can be used to detect properties of the original group.   Bergeron, Hagung and Wise  \cite[Theorem 1.10]{BergeronHaglundWise2011} proved that cocompact standard arithmetic subgroups of $\SO(1,n)$ are virtually compact special and for non-cocompact case they are virtually virtual retracts of virtually compact special groups by \cite[Lemma 6.3 and Theorem 7.4]{PS}. The objective of this paper is to use this result to prove that
these groups are conjugacy separable and to deduce this property for groups of units $\U( \Z H)$ of  group rings $\Z H$ of certain finite groups $H$ (see Theorem \ref{good}). We also show that finite quotients of the unit group $\U( \Z H)$ of these group rings determine  $\U( \Z H)$ up to isomorphism and hence the group ring itself. Note that the question of determining a group by its finite quotients (or profinite completion) is an old question in group theory, but had a new wind in 21-st century in geometric group theory.

A group $G$ is  termed conjugacy separable if given any two elements  $x$ and $y$ that are
non-conjugate  in  $G$, there exists some finite
quotient of $G$ in which the images of $x$ and $y$ are
 not conjugate.
In other words the group is  conjugacy separable if the conjugacy class of every element is closed in the profinite topology of $G$.
The notion
owes its importance
to the fact, first pointed out by Mal'cev \cite{Malcev1958}, that the conjugacy problem  has
a positive solution in finitely presented conjugacy separable group. 
E. Grossman \cite{Grossman1974} also observed  that the residual finiteness of
$Out(H)$ is equivalent to the conjugacy separability of $H$  for a finitely
generated group $H$ whose
elementwise inner automorphisms (i.e. automorphisms preserving conjugacy classes) are inner.  Moreover, the observation
of Grossmann
becomes especially important due to the result of Minasyan and Osin
\cite{MinasyanOsin2010} stating that every pointwise inner automorphism of a
torsion-free relatively hyperbolic group is inner (for torsion-free hyperbolic groups the result was obtained
independently by Bogopolsky and Ventura \cite{BogopolskiVentura2011}).

  Well-known classical examples of conjugacy  separable groups are
polycyclic-by-finite groups, free-by-finite groups and surface-by-finite groups. Moreover, fundamental groups of 3-manifolds are also conjugacy separable.  Minasyan  \cite{Minasyan2012} proved that right angled Artin groups are conjugacy separable. This result was used  by Minasyan and the  third author to show \cite{MinasyanZalesskii2015} that all hyperbolic virtually compact special groups (in the sense of D.Wise) are conjugacy separable.

The first objective of this paper is to extend the last result by relaxing the hyperbolicity hypothesis. More precisely we prove the following:

\begin{thm}\label{GVCSHRNRH}
Let $G$ be a group having  a subgroup of finite index which is a virtual retract of a compact special and toral relatively hyperbolic group. Then $G$ is conjugacy separable. (see Definitions 
\ref{virtspe}   and \ref{DefToral}).
\end{thm}

Arithmetic groups arise naturally as discrete subgroups of Lie groups, defined by arithmetic properties. By the fundamental theorem of Borel and Harish-Chandra an arithmetic subgroup  $\Gamma$ of  the special orthogonal group $\SO(1,n)$ is a lattice (i.e a discrete subgroup of finite covolume).
An arithmetic lattice $\Gamma$ of  $\SO(1,n)$ is said to be \emph{standard} (or of the simplest type) if it comes from a quadratic form $q$ defined over a totally real algebraic number field such that its completion for one real place is isomorphic to $\SO(1,n)$ and for all other real places is either positively or negatively defined. More precisely, let $R_k$ be the ring of integers of $k$. A {\it standard arithmetic lattice} $\Gamma$ of $SO(n,1)$ is a group commensurable with the subgroup $SO(q, R_k)$ of $R_k$-points of $SO(q)= \{X \in SL(n + 1, \C) \mid X^t q X = q \}$.  The construction of $\Gamma$ can be found in \cite[Section 6.4]{Morris2015}, for example. The  non-standard arithmetic lattices that come from Hermitian forms are not considered in this paper. Note also that for $n$ even, all arithmetic lattices of $\SO(1,n)$ are standard (see \cite[Remark 2.1]{KPV}).

As an application of \Cref{GVCSHRNRH} we obtain the following:

\begin{thm}\label{standard arithmetic}
Standard arithmetic lattices of special orthogonal groups $\SO(1,n)$ are conjugacy separable.
\end{thm}

It is known that arithmetic lattices having the  Congruence Subgroup Property (see \Cref{CSP} for definition) are not conjugacy separable (cf. \cite[Theorem 3]{Stebe1972},  \cite[Proposition 8.26 and remark after  it]{PlatonovRapinchuk1994}).
On the other hand, arithmetic subgroups of $\SL_2(\mathbb{C})$ do not  have   Congruence Subgroup Property  but are conjugacy separable.

  Based on this, the third author  conjectured at Banff conference on arithmetic groups (2013) that all arithmetic groups failing  Congruence Subgroup Property are conjugacy separable.  Theorem \ref{standard arithmetic} gives another evidence to this conjecture.

Note that all discrete subgroups of $\SO(1,3)$ or equivalently of $\SL_2(\mathbb{C})$ are virtually special and so conjugacy separability holds for all groups commensurable with them. We state this as  the following corollary.

\begin{cor} Let $G$ be the fundamental group of  a hyperbolic 3-orbifold. Then $G$ is conjugacy separable.	
\end{cor}

It is an open question weather non-standard arithmetic subgroups of $\SO(1,n)$ are virtually compact special, so the methods of this paper do not give the result for such arithmetic groups.

The results above allow us to prove conjugacy separability for the group of units $\U(\Z G)$ of the integral group rings of some finite groups $G$. These groups of units  have nice residual properties that allows to approach  the question to what extent a finitely generated group is determined by its finite quotients or equivalently by its profinite completion for $U(\Z G)$. 

For unit groups the question takes the following form:
\begin{quote}
	\textbf{Question 1}: If $G$ and $H$ are finite groups such that  $\U(\Z G)$ and $\U(\Z H)$ have the same finite epimorphic images up to isomorphism, are $\U(\Z G)$ and $\U(\Z H)$ isomorphic?
\end{quote}

 It is worth to observe that the isomorphism type of the integral group ring of a finite group is determined by  its group of units, i.e. if $G$ and $H$ are finite groups such that $\U(\Z G)$ and $\U(\Z H)$  are isomorphic then so are $\Z G$ and $\Z H$ \cite[Proposition~2.2]{Jespers2021}. However this does not imply that $G$ and $H$ are isomorphic \cite{Jespers2021}.

It is well-known that the profinite completion of two finitely generated groups are isomorphic  if and only if the families of their finite quotients coincide (see \cite[Corollary 3.2.8]{RibesZalesskii2010}  for example). Since the unit group is finitely generated we can therefore reformulate the question as follows: under which conditions does the profinite completion $\widehat{\U(\Z G)}$ of $\U(\Z G)$ determine $\U(\Z G)$ up to isomorphism?
Or more precisely, Question 1 can be reformulated as follows: 
\begin{quote}
	\textbf{Question $\widehat 1$}: Does
	$\widehat{\U(\Z G)}\cong \widehat{\U(\Z H)}$ imply $\U(\Z G)\cong \U(\Z H)$, for 
	$G$ and $H$ finite groups? 
\end{quote}

Note that in general it is very difficult to decide whether a residually finite group $G$ is determined by its profinite completion. For example, it is a long standing question of Remeslennikov whether a free group of finite rank is determined by its profinite completion within all finitely generated residually finite groups.  

To state our results on unit groups we need more terminology which we borrow from \cite[Definition~11.2.2]{JespersdelRio2016}. 

Let $A$ be a simple finite dimensional rational algebra. We say that $A$ is \emph{exceptional} if it is isomorphic to one of the following types:
\begin{enumerate}
	\item a non-commutative division algebra non-isomorphic to a totally definite quaternion algebra. (See \Cref{SectionPreliminaries} for the latter notion.)
	\item $M_2(D)$ with $D$ either $\Q$, an imaginary quadratic extension of $\Q$ or a totally definite quaternion algebra over $\Q$. 
\end{enumerate}

\bigskip Lattices  of the exceptional components of type (1) as well as in  all non exceptional components satisfy the celebrated Margulis Normal Subgroup Theorem \cite[Theorem 4, Introduction]{margulis}: a normal subgroup is either finite or of finite index. In particular these groups are Fab, i.e. every sugroup of finite index has finite abelianization (this notion comes from Galois theory \cite{Labute}).

\medskip

In contrast, lattices of  exceptional components of type (2) are virtually virtual retracts of compact special groups and so have  virtual cyclic retract property (VCR): every cyclic subgroup is a (not normal) semidirect factor of a subgroup of finite index. This property was first introduced by Long and Reid in \cite{LR},
however, implicitly they were investigated much earlier. Many important groups possesses this property: free groups, surface groups, hyperbolic 3-manifold groups, virtually special groups, Right angled Artin groups and Coxeter groups (see \cite[Corolary 1.6]{minasyan}). In fact, Minasyan proved \cite[Theorem 1.5]{minasyan} that VCR is equivalent to virtual abelian retract property.

Note that VCR is stable under direct product and commensurability (see \cite[Theorem 1.4]{minasyan} that allows us to deduce the following characterization:

\begin{thm}\label{characterization}
	Let $G$ be a finite group and $\Z G$ be its group ring. Then the following conditions are equivalent
	\begin{enumerate}
	
		\item[(i)] $U(\Z G)$ is virtually special;
		\item[(ii)] $U(\Z G)$ has virtual cyclic retract property;
		\item[(iii)] $U(\Z G)$ has virtual abelian retract property;
		\item[(iv)] every non-commutative simple quotient of its rational group algebra is either a totally definite quaternion algebra, or exceptional of type (2)
		\end{enumerate}
\end{thm}

There exists a cohomological version of  Question $\widehat 1$. According to Serre a  group $G$ is called good if the cohomology of $G$ coincide with the cohomology of its profinite completion in any finite module. More precisely, $G$ is good if the natural homomorphism $G\longrightarrow \widehat G$ induces an isomorphism of cohomology groups
$$H^n(\widehat G,M)\longrightarrow H^n(G,M)$$ for any finite $G$-module $M$. In general it is not easy to decide whether a given group is good, for example it is an open question whether mapping class groups are  good. 

Note that conjugacy separability is not preserved  by passing to subgroup  or overgroups of finite index in general (see \cite{CZ2,G}). A group where the conjugacy separability holds for all finite index subgroups is called \emph{hereditarily conjugacy separable}.   
We prove that for the groups $G$ from \Cref{characterization} the group of units $\U(\Z G)$  is good, hereditarily conjugacy separable and determined by its finite quotients.

\begin{thm}\label{good}
	Let $G$ be a finite group such that every non-commutative simple quotient of its rational group algebra is either a totally definite quaternion algebra, or exceptional of type (2). Then 
	\begin{enumerate}
		\item[(i)] 
		$\U(\Z G)$ is good. In particular, the profinite completion of any torsion-free subgroup of $\U(\Z G)$  of finite index is torsion-free.
		
		\item[(ii)] $\U(\Z G)$ is hereditarily conjugacy separable.
	\end{enumerate}
\end{thm}

\begin{thm}\label{Units}
	Let $G$ be a finite group such that every non-commutative simple quotient of its rational group algebra is either a totally definite quaternion algebra, or exceptional of type (2). Then the following conditions are equivalent for another finite group $H$:
	\begin{enumerate}
		\item $G\cong H$.
		\item $\Z G\cong \Z H$.
		\item $\U(\Z G)\cong \U(\Z H)$.
		\item The profinite completions of $\U(\Z G)$ and $\U(\Z H)$ are isomorphic.
		\item $\U(\Z G)$ and $\U(\Z H)$ have the same finite epimorphic images. 
	\end{enumerate}
\end{thm}

Finally note that Menny Aka \cite{A} proved that  the profinite completion determines an  arithmetic group having  Congruence Subgroup Property up to finitely many non isomorphic arithmetic groups. This shows that if $U(\Z G)$ has only one non-abelian not exceptional component, then  there are only finitely many groups $H$ such that $\widehat{U(\Z H)}\cong \widehat{U(\Z G)}$.

\bigskip 
The paper is organized as follows. In Section 2 we introduce the necessary terminology and notation. Section 3 is dedicated to the conjugacy separability, where we prove Theorems  \ref{GVCSHRNRH} and \ref{standard arithmetic}. Section 4  deals with units of group rings: the proofs of Theorems  \ref{good} and \ref{Units} can be found there. In Section 5 we recall the statement of the Congruence Subgroup Problem for arithmetic groups and units of group rings and use Theorem \ref{good} to bound the  virtual cohomological dimension of the congruence kernel  of the unit group of the group rings from Theorem \ref{good}.

\section*{Acknowledgement}
The third author thanks Ashot Minasyan and Andrei Rapinchuk  for  fruitful discussions during the work on this paper.

\section{Notation and preliminaries}\label{SectionPreliminaries}

Let $n$ be an non-zero integer. Then $\zeta_n$ denotes a complex primitive root of unity.
If $p$ is a prime integer then $v_p(n)$ denotes the valuation of $n$ at $p$, i.e. the maximum non-negative integer $k$ with $p^k$ dividing $n$.
If $m$ is another integer coprime with $n$ then $o_m(n)$ denotes the order of $n$ modulo $m$, i.e. the minimum positive integer $k$ with $n^k\equiv 1 \mod m$.

We use standard group theoretical notation, for example, if $g$ and $h$ are elements of a group $G$ then $g^h=h^{-1}gh$, $C_G(g)$ denotes the centralizer of $g$ in $G$, $g^G$ denotes the conjugacy class of $g$ in $G$ and for a subgroup $H$ of $G$, $N_G(H)$ denotes the normalizer of $H$ in $G$. We use $N:H$ for extension of groups and $G\rtimes H$ for split extensions, i.e. $N:H$ denotes a group $G$ with a normal subgroup $N$ such that $G/N\cong H$, and $G=N\rtimes H$ when $H$ is a normal complement of $N$ in $G$. The profinite completion of $G$ is denoted $\widehat{G}$ and if $X$ is subset of $G$ then $\overline{X}$ denotes the closure of the natural image of $X$ in $\widehat G$.

For a positive integer $n$, $C_n$, $D_n$ and $Q_n$ denotes the cyclic, dihedral and quaternion groups of order $n$; $S_n$ and $A_n$ denote the symmetric and alternating groups on $n$ symbols; and for a commutative ring $R$ we use the standard notation for the general linear group $\GL_n(R)$, the projective general linear group $\PGL_n(R)$, the special linear group $\SL_n(R)$ and the projective special linear group $\PSL_n(R)$. In case $R$ is a field with  $q$ elements then the latter groups are denoted by $\GL(n,q), \PGL(n,q), \SL(n,q)$ and $\PSL(n,q)$, respectively. Moreover, for $n$ small enough $\SG[n,m]$ denotes the $m$-th group of order $n$ in the GAP library of small groups, i.e. \verb+SmallGroup(n,m)+ in GAP terminology.

The \emph{Wedderburn decomposition} of a semisimple ring $A$ is its expression as a direct product of simple artinian rings. We refer to this simple factors as \emph{Wedderburn components} of $A$. The GAP package \texttt{Wedderga} \cite{Wedderga} contains some functions which compute the Wedderburn decomposition of groups algebras of finite groups. It is based on the methods introduced in \cite{OlivieridelRioSimon2004} and extended in \cite{Olteanu2007}. A fundamental notion of these methods is that of strong Shoda pair of a group $G$. Associated to every strong Shoda pair $(H,K)$ of $G$ there is a Wedderburn component of $\Q G$ whose structure can be explicitly given (see  \cite[Section~3.5]{JespersdelRio2016} for details). This will be used in the proof of \Cref{ExceptionalComponents}.

The group of units of a non-necessarily commutative ring $R$, is denoted by $\U(R)$. 
In case $R$ is  a subring of a finite dimensional rational algebra $A$ then $\SL_n(R)$ denotes the group of elements of $M_n(R)$ with reduced norm $1$ in $M_n(A)$.

If $E/F$ is a Galois extension of fields of order $n$ with Galois group generated by $\sigma$ and $a\in \U(F)$ then $(E/F,a)$ denotes the cyclic algebra, i.e. $$(E/F,a)=F[u:u^{-1}xu=\sigma(x) (x\in F), u^n=a].$$

The \emph{quaternion algebras} over a field $F$ of characteristic different from $2$ are defined as
$$\quat{a,b}{F} = F[i,j:i^2=a,j^2=b,ji=-ij],$$
for $a,b\in \U(F)$.
This quaternion algebra is said to be \emph{totally definite} when $F$ is a totally real number field and $a$ and $b$ are totally positive.
We use classical notation for the Hamiltonian quaternion algebra: $\HQ(F)=\quat{-1,-1}{F}$.

If $q$ is a quadratic form defined over a commutative ring $R$ then  $\SO(q;R)$ denotes the corresponding special orthogonal group. For positive integers $n$ and $m$, $$\SO(n,m)=\SO\left(\sum_{i=1}^n x_i^2-\sum_{j=1}^m x_j^2;\R\right).$$

\begin{deff} Let $G\subseteq  GL_n(\mathbb{C})$ be a linear algebraic group defined over $\mathbb{Q}$. A subgroup $\Gamma\subseteq G $ is \emph{arithmetic} if it is commensurable with $G_{\mathbb{Z}} = G\cap  GL_n(\mathbb{Z})$, i.e., if $\Gamma\cap G_{\mathbb{Z}}$ has finite index both in $\Gamma$ and in $G_{\mathbb{Z}}$.
\end{deff}

Note that by the fundamental theorem of Borel and Harish-Chandra if $G$ is semisimple then  $\Gamma$ is a lattice in $G$ (i.e. a discrete subgroup of finite covolume).

We give now a description of the class of arithmetic lattices to which the main theorem applies. 

Let $f$ be a quadratic form of signature $(n, 1)$ in $n+ 1$ variables with coefficients
in a totally real algebraic number field $K \subset \R$ satisfying the following condition:
 For every nontrivial (i.e., different from the identity) embedding $\sigma : K \rightarrow \R$
the quadratic form $f^\sigma$
is positive definite.
Let $A$ denote the ring of integers of $K$. We define the group $\Gamma := SO(f, A)$
consisting of matrices of determinant 1 with entries in $A$ preserving the form $f$. Then $\Gamma$ is a
discrete subgroup of $SO(f, \R)$. Moreover, it is a lattice. Such groups $\Gamma$ (and groups commensurable to them) are called \emph{standard arithmetic subgroups} of  $SO(n, 1$) (or \emph{of the simplest type}).
Note that if $\Gamma$  is non-uniform  or $n$ is even, then   $\Gamma$ is a standard arithmetic lattice (see \cite[Remark 2.1]{KPV}).

In general, for n odd, $\SO(1,n)$ has another type of cocompact arithmetically defined subgroups, given as the groups
of units of appropriate skew-Hermitian forms over quaternionic algebras and 
other families of arithmetic lattices exist for n = 3 and n = 7. For $n\neq 3$ It is not known whether these groups are virtually compact special, so they are not considered in this paper.

\section{Conjugacy separability}\label{sec:proof}

Recall that the profinite topology on a group $G$ is a topology for which the set of all normal subgroups of finite index is taken as a base of the identity.

We say that an element $g$ of a group $G$ is \emph{conjugacy distinguished} in $G$ if  its conjugacy class $g^G$  is closed in the profinite topology of $G$. 
For residually finite $G$ this exactly means $g^{\widehat{G}}\cap G= g^G$. 
Note that $G$ is \emph{conjugacy separable} if and only if every element of $G$ is conjugacy distinguished.  

Recall that a subgroup $H$ of $G$ is \emph{separable} in $G$ if for every $g\in G\setminus H$ there is a normal  subgroup of finite index $K$ in $G$ such that $g\not\in KH$. 

The following proposition follows directly from the proof of \cite[Proposition~2.1]{ChagasZalesskii2010}, having in mind that as $H$ is assumed to be hereditarily conjugacy separable we have that $C_H(x^m)$ is dense in $C_{\widehat H}(x^m)$ \cite[Corollary~12.3]{Minasyan2012}.

\begin{prop}\label{prop:C-Z_crit} Let $H$ be a normal subgroup of index $m \in \N$ in a group $G$ and let $x \in G$. Suppose that $H$ is hereditarily conjugacy separable
	and  the following conditions hold:
	\begin{itemize}
		\item[(i)] $x$ is conjugacy distinguished in $C_G(x^m)$;
		\item[(ii)] each finite index subgroup of  $C_G(x^m)$ is separable in $G$.
	\end{itemize}
	Then $x$ is conjugacy distinguished in $G$.
\end{prop}

\begin{deff}\label{virtspe} A group $G$ is said to be virtually compact special if it has a finite index subgroup isomorphic to the fundamental group of a compact special cube complex (see \cite{H-W-1} for details).

\end{deff} 
\begin{deff}\label{DefToral} A group $G$ is called  toral relatively hyperbolic
if $G$ is torsion-free, and hyperbolic relative to a finite set of finitely generated abelian
subgroups.\end{deff}

\begin{deff} A subgroup $H$ of a group $G$ is said to be malnormal if $H\cap H^g=1$ for any $g\in G\setminus H$.\end{deff}

\begin{lemma}\label{lem:inf_order} Let $G$ be  a virtually compact special toral relatively hyperbolic group.  Let $g \in G$ be an element of infinite order. Then $g$ is conjugacy distinguished in $G$.
\end{lemma}

\begin{proof}  Let $g\in G$ be an infinite order element. 
Let  $H$  be  a finite index subgroup of  $G$ such that $H$ is a compact special and toral relatively hyperbolic. 
By  \cite[Corollary 2.2]{Minasyan2012}, $H$ is hereditarily conjugacy separable. 

Since $g$ is of  infinite order, $g^m\in H$ for some   integer $m$.  
Let $\gamma$ be a element of $\widehat{G}$ such that $g^\gamma\in G$. Observe that $\widehat{G}= G\widehat{H}$, so we can write $\gamma= \gamma_0\delta$, where  $\gamma_0\in G$ and $\delta\in\widehat{H}$.  Now substituting $g$ by $g^\delta$, we can suppose that $\gamma\in \widehat{ H}$. Since $H$ is conjugacy separable, $g^m$ and  $(g^\gamma )^m$ are conjugate in $H$, so we may assume that $g^m = (g^m)^\gamma$ and so $\gamma\in C_{\widehat H}(g^m)$.

We have two cases to consider: \\

\noindent {\bf Case 1:}  Suppose that $g^m$ is  conjugate to an element of a parabolic  subgroup of  $H$.  We may assume then that $g^m$ is in some parabolic subgroup $P$, and since parabolic subgroups are virtually malnormal (See \cite[Remark 2.2]{GrovesManning2020}),  $C_H(g^m)$ is   contained    in $P$. Then $\overline{C_H(g^m)}\leq \overline P$ and since  $\overline{C_H(g^m)}= C_{\widehat H}(g^m)$ by \cite[Proposition~3.2]{Minasyan2012}, we may deduce that $\gamma$ belongs to the closure of this parabolic subgroup in $\widehat G$.  One deduces from the cyclic subgroup separability of virtually special groups (see \cite[Theorem 1.4]{ChagasZalesskii2013}) that the closure of a free abelian group in $\widehat H$ is a free profinite abelian group, i.e. the closure coincides with its profinite completion. Then from conjugacy separability of virtually abelian groups one deduces  the result.

\noindent {\bf Case 2:} Suppose now that $g^m$ is not  conjugate to an element of any parabolic subgroup of $H$. So $g^m$ is a hyperbolic element and,  by  \cite[Theorem 4.3]{Osin2006} $C_G(g^m)$ is virtually cyclic and so is conjugacy separable. By \cite[Corollary 12.3 and Corollary 12.2]{Minasyan2012}, $C_G(g^m)$ is dense in $C_{\widehat G}(g^m)$, and by \cite[Theorem 1.4]{ChagasZalesskii2013}, its closure coincides with the profinite completion. Thus  we deduce from conjugacy separability of  $C_G(g^m)$ that $g$ and $g^\gamma$ are conjugate in $C_G(g^m)$ in this case.
This finishes the proof.
\end{proof}

\begin{rem}\label{not cocompact} In the proof of Lemma \ref{lem:inf_order} we actually used less than given in the hypothesis. The lemma is still valid if we assume that $H$ is a retract of a virtually compact special group (or equivalently  is a retract of a right angled Artin group). Indeed, what we use is that $H$ is conjugacy separable, cyclic subgroup separable and that $g$ and $g^\gamma$ are conjugate for any $\gamma\in C_{\widehat G}(g^m)$. The first two properties also hold for retracts of virtually compact special groups.  Now $C_H(g^m)$ is virtually cyclic and  is dense in $C_{\widehat H}(g^m)$. But $\gamma\in C_{\widehat H}(g^m)$ and as $C_{\widehat H}(g^m)\leq C_{\widehat H}(g)$ we deduce that $g=g^\gamma$. 
\end{rem}

\begin{prop}\label{relhypvirtualspecial} Let  $G$ be a virtual compact special  toral relatively hyperbolic group.  Suppose that  $G$ splits as a semidirect product $G=H \rtimes \langle x \rangle$, where $H$ is torsion-free and $x$ has finite order.  Then any torsion element of
	 $G$ is conjugacy distinguished.
\end{prop}
\begin{proof}
We shall use induction on the index $|G: H|=n$ to prove that every torsion element $g$ of $G$ is conjugacy distinguished. This is clear if $n=1$ or $g=1$, so we suppose $g\ne 1$, $n>1$, and the induction hypothesis.

Set $R=H \langle g \rangle$. 
If $|R:H|<n$ then by the  induction hypothesis $R$ has the statement of the proposition, so $g$ is conjugacy distinguished in $R$. But then  $g$ is conjugacy distinguished in $G$, as
$|G:R| \le |G:H|<\infty$.

Therefore we can assume that $|H: R| = n =|G:H|$. It follows that $G= R$, i.e., $G=H\langle g\rangle \cong   H \rtimes \langle g\rangle$, as $H$ is torsion-free and  $g$ has finite order (which must then be equal to $n$). We will now consider two cases.

{\it Case a:}  Suppose that $n$  is   a prime number $p$. Then $g$ is conjugacy distinguished in $G$ by \cite[Corollary 3.9.]{MinasyanZalesskii2015}.

{\it Case b:}  $n$ is a composite number. 
Thus $n=mp$ for some prime $p$ with $1< m< n$.

Let $K=H \langle g^m \rangle$. Then  $K\cong H \rtimes C_p$. Thus $K$ is hereditarily conjugacy separable by the induction hypothesis, as $|K:H|=m<n$. Evidently, $K\lhd G$ and $|G:K|=m$. By \cite[Theorem 1.1]{MinasyanOsin2012} $C_G(g^m)$ is relative quasiconvex and
since all  relative quasiconvex subgroups are separable by \cite[Theorem 4.7 combined with Theorem C]{GrovesManning2020}, it follows that every finite index subgroup of $C_G(g^m)$ is separable in $G$, so to use Proposition \ref*{prop:C-Z_crit} it remains to check that $g$ is conjugacy distinguished in $C_G(g^m)$.

Set $K_1=C_G(g^m) \cap H$, and observe that $C_G(g^m)=K_1 \langle g \rangle \cong K_1 \rtimes (\Z/n)$. Moreover, 
$K_1$ is   toral relatively  hyperbolic virtually compact special  as $|C_G(g^m):K_1|=n<\infty$ and $C_G(g^m) $ is  toral relatively  hyperbolic virtually compact special  by   \cite[Theorem 1.1]{MinasyanOsin2012}.

To verify that $g$ is conjugacy distinguished in $C_G(g^m)$, consider any element $h \in C_G(g^m)$ which is not conjugate to $g$ in $C_G(g^m)$.
Since $g^m$ is central in $C_G(g^m)$, we can let $L$ be the quotient of  $C_G(g^m)$ by $\langle g^m \rangle$, and let $\phi:C_G(g^m) \to L$ denote the natural epimorphism.

Clearly $\phi(K_1) \cong K_1$,  since  $K_1 \cap \ker \phi=\{1\}$.
Therefore $\phi(K_1)$ is torsion-free and $L=\phi(K_1) \langle
\phi(g) \rangle \cong K_1 \rtimes (\mathbb{Z}/m\mathbb{Z})$,  again   $L $ is virtually compact special groups since this class is closed  for finite index overgroups. Consequently, $L$ is
hereditarily conjugacy separable by the induction hypothesis, as
$|L:K_1|=m<n$.
Let us again consider two  subcases.

{\it Subcase b.1:} Suppose that $\phi(g)$ and $\phi(h)$ are not conjugate in $L$. Then there is a finite group $M$ and a homomorphism
$\psi:L \to M$ such that $\psi(\phi(g))$ is not conjugate to $\psi(\phi(h))$ in $M$. Thus the homomorphism $\eta=\psi\circ \phi:C_G(g^m) \to M$ will distinguish the conjugacy classes of $g$ and $h$, as required.

{\it Subcase b.2:} Assume that $\phi(g)$ is conjugate to $\phi(h)$ in $L$. Since $\ker \phi \subseteq \langle g \rangle$, we can deduce that
there is $y \in C_G(g^m)$ such that $ygy^{-1}=zh$, for some $z \in \langle g\rangle$.

Now, $z \neq 1$, since we assumed that $h$ is not conjugate to $g$ in $C_G(g^m)$. Therefore $1 \neq \xi(z)=z$, where $\xi: C_G(g^m) \to \langle g \rangle$ is the natural retraction (coming from the decomposition of
$C_G(g^m)$ as a semidirect product of $K_1$ and $\langle g\rangle$). Recalling that $\langle g\rangle$ is abelian, we see that $\xi(h)=\xi(ygy^{-1})=\xi(z)\xi(h)$. Therefore $\xi(h)$ is not conjugate to $\xi(g)$ in the finite cyclic group
$\langle g \rangle$. Thus we have distinguished the conjugacy classes of $g$ and $h$ in this finite quotient of $C_G(g^m)$.

Subcases b.1 and b.2 together imply that $g$ is conjugacy distinguished in $C_G(g^m)$. Therefore we have verified all of
the assumptions of Proposition \ref{prop:C-Z_crit} (for $G$ and the finite index normal subgroup
$K \lhd G$), so we can apply this proposition to deduce that $g$ is conjugacy distinguished in $G$. Thus Case 2 is completed.
 This finishes the proof of the proposition.
\end{proof}

\begin{rem}\label{prop:not cocompact} Proposition \ref{relhypvirtualspecial} is still valid if we assume that $H$ is a retract of a virtually compact special group (or equivalently  is a retract of a right angled Artin group). Indeed,  for item (a) in the proof  one has to use the combination of  Theorem 1.5, Proposition 3.8 and Lemma 2.5 in \cite{MinasyanZalesskii2015}, instead of \cite[Corollary 3.9]{MinasyanZalesskii2015}.
Suppose now we are in case (b). Then for the whole argument to work we just need to show that $C_G(g^m)$ is virtualy compact special within our hypothesis. But  $G'_0$ is virtually compact special and so $C_{G'_0}(g^m) $ is  toral relatively  hyperbolic virtually compact special  by   \cite[Theorem 1.1]{MinasyanOsin2012}. As $H$ is torsion-free, $g^m$ does not belong to any edge group of $G'_0$ and so $C_{G'_0}(g^m)=C_G(g^m)$ as required.  
\end{rem}

We can now prove \Cref{GVCSHRNRH} and \Cref{standard arithmetic}:

\begin{proofof}\emph{\Cref{GVCSHRNRH}}.	
By hypothesis $G$ has a subgroup $H$ of finite index which is toral relatively hyperbolic  and is a retract of a virtually compact special group. We   may assume $H$ to be normal taking its core; in particular  $H$ is torsion-free and so is hereditarily conjugacy separable by Lemma \ref{lem:inf_order} combined with Remark \ref{not cocompact}. We use induction on $n=[G:H]$.
Let $g$ be an arbitrary element of $G$, we need to prove that $g$ is conjugacy distinguished. If   $g$ has infinite order the result follows from \Cref{lem:inf_order} and Remark \ref{not cocompact}.
Suppose that   $g$ has finite order and $g^\gamma\in G$ for some $\gamma\in \widehat{G}$. Since $\widehat{G}= \widehat{H}G$, one has $\gamma=g_0\eta$ for some $g_0\in G$ and $\eta\in \widehat H$, so conjugating $g$ with $g_0$ we  may assume that $\gamma=\eta\in\widehat{H}$.
Set $K=H \langle g \rangle$.  Then $K= H \rtimes \langle g \rangle$.  Then by Proposition \ref{relhypvirtualspecial} and Remark \ref{prop:not cocompact} $g$ is conjugacy distunguished in $K$, i.e. there exists $h\in H$ such that $g^h=g^\gamma$ as needed. 
\end{proofof}

\begin{proofof}\emph{\Cref{standard arithmetic}}.
Observe that any torsion-free lattice in \SO(1,n) is  the fundamental group of a finite volume hyperbolic manifold and is therefore relatively hyperbolic to cusps which are free abelian groups.	
Then the result follows from \Cref{GVCSHRNRH}, since in the cocompact case by \cite[Theorem 1.10]{BergeronHaglundWise2011} a standard arithmetic subgroup of	$\SO(1,n)$ is virtually compact special and in the non-cocompact  case,  by \cite[Lemma 6.3 and Theorem 7.4]{PS}, a standard arithmetic subgroup of	$\SO(1,n)$ is a virtually virtual retract of a virtually compact special group.
	\end{proofof}

\begin{rem} \label{quaternions}	Note that $Spin(1,5)\cong \SL(2, \HQ)$, where $\HQ$ is the real algebra of Hamiltonian quaternions and so   arithmetic subgroups of it are commensurable with arithmetic subgroups of $\SO(1,5)$. Moreover, they are commensurable with a standard arithmetic subgroup of 
	$\SO(1,5)$. Indeed, let $H$ be a torsion-free arithmetic subgroup of $\SL(2, \HQ)$ contained in $\SO(1,5)$. Then $H$ is the fundamental group of a hyperbolic manifold $M$. But  $M$ is not closed, since the fundamental group of a closed hyperbolic manifold is hyperbolic, but $H$ contains $\Z\times\Z$. Then by \cite[Proposition 6.4.2]{Morris2015}, $H$ is standard.

\end{rem}

From Remark \ref{quaternions} we deduce the following:

\begin{cor} Let $\mathbb{H}$ be an algebra of quaternions. Standard arithmetic subgroups of $\SL(2, \mathbb{H})$ are conjugacy separable. 
\end{cor}

\section{Unit groups}

In this section $G$ is a finite group, $\Z G$ denotes its integral group ring and $\U(\Z G)$ the group of units of $\Z G$.
The aim of this section is to prove Theorems \ref{good} and \ref{Units}.
So we assume that $G$ satisfies the hypothesis of \Cref{good}, i.e. every non-commutative simple quotient of $\Q G$ is either totally definite quaternion or exceptional of type 2.

We start with the proof of  \Cref{good}.

\begin{proofof}{\it \Cref{good}}  	
Note that $\Z G$ is an order in $\Q G$ and $\Q G=\prod_{i=1}^n A_i$ with $A_i=\Q Ge_i$, where $e_1,...,e_n$ are the primitive central idempotents of $\Q G$. 
We order the $A_i$'s so that $A_i$ is either commutative or totally definite quaternion if and only if $i>m$. For every $i=1,\dots,n$ we fix an order $R_i$ in $A_i$ (for example we can take $R_i=\Z Ge_i$).	
	
	(i) It is well known that goodness is closed for commensurability and direct products (cf. \cite[Lemma 3.2 and Proposition 3.4]{GrunewaldJaikinZalesskii2008}). 
	It is also well known that, $\U(\Z G)$ is commensurable with $A\times \prod_{i>m} R_i^1$ where $A$ is a finitely generated abelian group  \cite[Remark~4.6.10 and Proposition~5.5.6]{JespersdelRio2016}). 

	Therefore we need to prove that  if $R$ is an order in $D$ with $M_2(D)$ an exceptional component of type (2) then $\SL_2(R)$ is good.  
	The goodness of $\SL_2(\Z)$ is the subject of \cite[Exercise 2, page 16]{Serre1997}. 
	By \cite[Corollary 4.3]{GrunewaldJaikinZalesskii2008}, if  $D$ is an imaginary quadratic field then $\SL_2(R)$  is good. Thus we need only to consider  the case where $D$ is a totally definite quaternion algebra over $\Q$.  
	
	By Remark \ref{quaternions},  $Spin(1,5)\cong \SL(2, \HQ)$, where $\HQ$ is the real algebra of Hamiltonian quaternions and so   $\SL_2(R)$ is commensurable with arithmetic  lattices of $\SO(1,5)$. By \cite[Lemma 6.3 and Theorem 7.4]{PS} a standard arithmetic subgroup of	$\SO(1,n)$ is a virtually virtual retract of a virtually compact special group and by   Haglund and Wise proved in \cite{H-W-1} that any such group is a virtually virtual retract of right angled Artin group.  By \cite[Proposition 3.8]{MinasyanZalesskii2015}  virtually virtual retracts of right angled Artin group  are good. 
	
	It remains to explain that $\SL_2(R)$ is commensurable with a standard arithmetic subgroup of 
	$\SO(1,5)$. Let $H$ be a torsion-free finite index subgroup of $\SL_2(R)$ contained in $\SO(1,5)$. Then $H$ is the fundamental group of a hyperbolic manifold $M$. Then $M$ is not closed, since the fundamental group of a closed hyperbolic manifold is hyperbolic, but $H$ contains $\Z\times\Z$. Then by \cite[Proposition 6.4.2]{Morris2015}, $H$ is standard. This finishes the proof of (i).
	
	(ii)   Taking $R_i = \Z Ge_i$  the map $x\mapsto (xe_1,...,xe_n)$ is an embedding of $\Z G$ in $\prod_i R_i$. Then $\U(\Z G)=\Z G \cap \prod_i \U(R_i) \subseteq \prod_i \U(R_i) \subseteq \prod_i R_i$. By \Cref{standard arithmetic} combined with \Cref{quaternions}, we have that $\U(R_i)$ is hereditarily conjugacy separable. By \cite[Theorem 1.2]{Ferov2016} a direct product of hereditarily conjugacy separable groups is hereditarily conjugacy separable, hence so is $\U(\Z G)$. This finishes the proof of \Cref{good}.
\end{proofof}

For the proof of \Cref{good} it was enough to use minimal properties of the simple components of $\Q G$. However for the proof of \Cref{Units} we need to analyze more carefully the possible simple components $A_i$ and the projections $Ge_i$ of $G$ in each $A_i$. Observe that each $A_i$ is either exceptional of type 2 or a division algebra. The possible exceptional components of type 2 and their respective groups $Ge_i$ where classified in \cite{EiseleKieferVanGelder2015}. On the other hand the finite subgroups of division algebras and the minimal division algebras containing them were classified by Amitsur in \cite{Amitsur1955}. These two classifications will be the main tool for the proof of \Cref{Units}. 
As a technicality we start with a list of groups having a non-allowed (i.e. not satisfying Theorem \ref{characterization} (iii)) simple component on their rational group algebra.

\begin{table}[h]
	$$
	\matriz{{llc}
		\hline
		G & \text{Structure} &  \text{Component of } \Q G \\\hline\hline
		\SG[12,3] & A_4=\PSL(2,3) & M_3(\Q) \\
		\SG[20,3] & C_5\rtimes C_4 & M_4(\Q) \\
		\SG[24,12] & S_4=\PGL(2,3) & M_3(\Q) \\
		\SG[32,6] & C_2^3\rtimes C_4 & M_4(\Q) \\
		\SG[36,9] & C_3^2\rtimes C_4 & M_4(\Q) \\
		\SG[36,10] & S_3\times S_3 & M_4(\Q) \\
		\SG [60,5] & A_5=\PSL(2,5) & M_4(\Q) \\
		\SG[64,138] & (C_2^4\rtimes C_2)\rtimes C_2 & M_4(\Q) \\		
		\SG [72,40] & S_3^2 \rtimes C_2 & M_4(\Q) \\
		\SG [80,49] & C_2^4\rtimes C_5 & M_5(\Q) \\
		\SG [120,34] & S_5 & M_4(\Q) \\
		\SG [288,1025] & (((C_2^4)\rtimes C_3)\rtimes C_2)\rtimes C_3 & M_6(\Q) \\
		\SG[320,1581] & (((C_2\times Q_8):C_2):C_5).C_2 &  M_5(\Q) \\
		\SG [360,118] & A_6=\PSL(2,9) & M_5(\Q) \\ 
		\hline
	}
	$$
	\caption{\label{ExcludedGroups} The second column displays structural information of the group $G$ displayed in the first column. The last column displays one Wedderburn component of $\Q G$.}
\end{table}

The verification of the information displayed in \Cref{ExcludedGroups} can be done using the GAP package \texttt{Wedderga} \cite{Wedderga}. 
For example, if $G=\SG[80,49]$ then 
$$\Q G\cong \Q \oplus \Q(\zeta_5) \oplus 3 M_5(\Q).$$ 
This can be verified with the following GAP calculation 
\begin{verbatim}
LoadPackage("wedderga");;
G:=SmallGroup(80,49);;QG:=GroupRing(Rationals,G);;
WedderburnDecompositionInfo(QG);                    
[ [ 1, Rationals ], [ 1, CF(5) ], [ 5, Rationals ], [ 5, Rationals ], 
[ 5, Rationals ] ]
\end{verbatim}
The reader can verify the information displayed in \Cref{ExcludedGroups} with similar calculations.

The exceptional components of type (2) that can occur as Wedderburn components of the rational group algebra of a finite group are classified in \cite{EiseleKieferVanGelder2015}, as well as the finite groups which have such exceptional components in their rational group algebra as a faithful Wedderburn component (i.e. the group embeds faithfully on the given exceptional component). While the number of such exceptional components is finite (actually 7), the dimension of the exceptional components of type (1) which occur as Wedderburn components of the rational group algebra of a finite group is unbounded. This follows from Amitsur's characterization of finite subgroups of division rings \cite{Amitsur1955} and specially from its Theorem 3. For our proof of Theorem~\ref{Units} we will need the following proposition of independent interest.

\begin{table}[h]
$$\matriz{{cll}
\hline A & G \\\hline\hline
\HQ(\Q) & Q_8 \\\hline
\quat{-1,-3}{\Q} & Q_{12} \\\hline
\HQ(\Q(\sqrt{2})) & G=Q_{16} \\\hline
\quat{-1,-3}{Q(\sqrt{3})} & G=Q_{24} \\\hline
\quat{\zeta_{2^{n-1}},-3}{\Q(\zeta_{2^{n-1}})} & G=C_3\rtimes C_{2^n}, (n\ge 4), 
C_{2^n} \text{ acting on } C_3  \text{ by inversion} \\\hline
\HQ(\Q(\zeta_m)) & G=Q_8\times C_m, 
m \text{ and } o_m(2) \text{ odd}\\\hline\hline
M_2(\Q) & S_3, D_8, D_{12} \\\hline
M_2(\Q(i)) & C_4\times S_3, \SG[16,6]=C_8\rtimes C_2,  \\
& \SG[16,13]=(C_4\times C_2)\rtimes C_2, \SG[32,11]=(C_4\times C_4)\rtimes C_2 \\\hline
M_2(\Q(\sqrt{-2})) & \SG[16,8]=C_8\rtimes C_2 \\\hline
 & C_3\times S_3, C_3\times D_8, C_3\times Q_8, C_6\times S_3,  \\
M_2(\Q(\sqrt{-3})) & \SG[24,8]=(C_6\times C_2)\rtimes C_2,
\SG[36,6]=C_3\times (C_3\rtimes C_4),  \\
& \SG[72,30]=C_3\times ((C_6\times C_2)\rtimes C_2) \\\hline
& \SG[32,8]=(C_2\times C_2).(C_4\times C_2), \SG[32,44]=(C_2\times Q_8)\rtimes C_2, \\
M_2(\HQ(\Q)) & \SG[32,50]=(C_2\times Q_8)\rtimes C_2, Q_8\times S_3, \\
& \SG[64,137]=((C_4\times C_4)\times C_2)\rtimes C_2 \\\hline
M_2\quat{-1,-3}{\Q} & \SG[48,18]=C_3\rtimes Q_{16}, 
\SG[48,39]=(C_4\times S_3)\rtimes C_2.\\\hline
}$$
\caption{\label{EC} The algebras and groups satisfying the hypothesis of \Cref{ExceptionalComponents}. The cases where $A$ is a division algebra are above the double line.}
\end{table}

\begin{prop}\label{ExceptionalComponents}
Let $A$ be a non-commutative rational algebra generated over $\Q$ by a finite subgroup $G$ of $\U(A)$. Suppose that every non-exceptional simple component of $\Q G$ is a division algebra (equivalently either a field or a totally definite quaternion algebra). Then $A$ and $G$ are as in \Cref{EC}.

In particular, $G$ has an abelian normal subgroup $A$ such that $G/A$ has exponent dividing $4$.
\end{prop}

\begin{proof}
The hypothesis implies that for every epimorphic image $H$ of $G$, all the non-exceptional simple components of $\Q H$ are division algebras. In particular, $G$ has not an epimorphic image isomorphic to one of the groups in \Cref{ExcludedGroups}. 
	
\textbf{Case 1}: Suppose that $A$ is a division algebra. 

In this case we consider separately the different options from Amitsur classification of finite subgroups of division rings \cite{Amitsur1955} as presented in  \cite{ShirvaniWehrfritz1986}.

\textbf{Subcase 1.1}: Suppose that $G$ is a Z-group, i.e. every Sylow  subgroup of $G$ is cyclic. Then $G$ is in one of the cases b) or c) in  \cite[Theorem~2.1.5]{ShirvaniWehrfritz1986} (case a) is excluded because as $A$ is non-commutative, $G$ is non-abelian).

In case b) we have $G=\GEN{a}_m\rtimes \GEN{b}_4$ with $m$ odd and $b$ acting on $\GEN{a}$ by inversion. 

Then $H=G/\GEN{b}=D_{2m}$ and one simple component of $\Q G$ is of the form $M_2(\Q(\zeta_m+\Q(\zeta_m^{-1})))$. By assumption this must be exceptional and as the center is totally real it must be $\Q$, so that $\zeta_m+\zeta_m^{-1}\in \Q$. As $m$ is odd and different from $1$ it follows that $m=3$. Therefore $G=\GEN{a}_3\rtimes \GEN{b}_4=Q_{12}$ and $A=\quat{-1,-3}{\Q}$. The latter is obtained with the strong Shoda pair $(\GEN{a,b^2},1)$.

In case c), $G=G_0\times G_1\times \dots \times G_s$ with $s\ge 1$, $G_0$ cyclic, $\gcd(G_i,G_j)=1$ for $i\ne j$ and for each $i=1,\dots,r$  we have $G_i=C_{p^a}\rtimes (C_{q_1^{n_1}}\times \dots \times C_{q_r}^{n_r})$ where $p,q_1,\dots,q_r$ are distinct primes and if $q_i^{k_i}$ is the order of the kernel of the action of $C_{q_i^{n_i}}$ on $C_{p^a}$ then for every $j=1,\dots,r$ we have 
\begin{enumerate}
	\item $k_i<n_i$,
	\item  $v_{q_j}\left(o_{\frac{|G|}{|G_i|}}(p) \right) < o_{q^{k_j}}(p)$, 
	\item if either $q_i$ is odd, or $q=2$ and $p\equiv 1 \mod 4$, then $p\not\equiv 1 \mod q_i^{k_i+1}$,
	\item if $q_i=2$ and $p\equiv -1 \mod 4$ then either $k_i=1$ or $p\not\equiv -1 \mod 2^{k_i}$.
\end{enumerate}

Fix $i=1,\dots,s$ and $G_i=C_{p^a}\rtimes (C_{q_i^{n_1}}\times \dots \times C_{q_r^{n_r}})$, as above. Then $G_i$ has a unique central cyclic subgroup $N$ of order $q_1^{k_1}\dots q_r^{k_r}$ and  $H=G_i/N=\GEN{A}_{p^m}\rtimes \GEN{B}_{q_1^{d_i}\cdots q_r^{d_r}}$ with $d_i=n_i-k_i>0$ for each $i$ and faithful action.  Then $C_{p^m}$ is a maximal abelian subgroup of $H$ and hence $\Q H$ has epimorphic image isomorphic to $M_{q_1^{d_1}\cdots q_r^{d_r}}(F)$ where $F$ is the fixed field of the automorphism of $\Q(\zeta_{p^m})$ given by $\sigma(\zeta_{p^m})=\zeta_{p^m}^t$ if $A^B=A^t$. As this component must be exceptional we have $q_1^{d_i}\cdots q_r^{d_r}=2$ and $F$ is either $\Q$ or imaginary quadratic. This shows that $s=1$, $G_1=C_{p^m}\rtimes C_{2^n}$ with $C_{2^n}$ acting on $C_{p^m}$ by inversion, so that $k_1=n-1$ and $F=\Q(\zeta_{p^m}+\zeta_{p^m}^{-1})$. The argument of the previous case shows that $p^m=3$. By condition (3) we have $k_1\ge 3$, so that $n\ge 4$.
Moreover, $G$ also have an epimorphic image isomorphic to $K=G_0\times H$ and if $m=|G_0|$ then $\Q K$ has a simple component isomorphic to $M_2(\Q(\zeta_m))$. As $m$ is odd and $\Q(\zeta_m)$ must be either $\Q$ or imaginary quadratic and coprime with $3$ it follows that $m=1$. Thus $G=C_3\rtimes C_{2^n}$ with action by inversion and $n\ge 4$.
Moreover the unique Wedderburn component of $\Q G$ on which $G$ embeds faithfully is obtained with the strong Shoda pair $(C_3\times C_{2^{n-1}},1)$ so that $A=\quat{\zeta_{2^{n-1}},-1}{\Q(\zeta_{2^{n-1}})}$. 

\textbf{Subcase 1.2}. Now we suppose that $G$ is not a Z-group. Then $G$ is as in cases b) or c) in \cite[Theorem~2.1.4]{ShirvaniWehrfritz1986}. However case b)i), b)iv) and c) are excluded because $S_4$, $\PSL(2,3)$ and $\PSL(2,5)$ are epimorphic images of the binary octaedral group, $\SL(2,3)$ and $\SL(2,5)$, respectively, appearing  in \Cref{ExcludedGroups}. So $G$ is in one of the cases b)i), b)ii) or b)iii). 

In case b)ii), $G=Q_{4m}$, a quaternion group with $m$ even. Then $G$ has a central subgroup $N$ of order $2$ with $H=G/N=D_{2m}$ a dihedral group of order $2m$. 
Then $\Q H$ has a simple component isomorphic to $M_2(\Q(\zeta_m+\zeta_m^{-1}))$ and hence $m\in \{2,4,6\}$. In the three cases $G$ has a cyclic subgroup $H$ of index $2$ which yields a strong Shoda pair $(H,1)$ and the corresponding  Wedderburn component is the cyclic algebra $A=(\Q(\zeta_{2m}),\sigma,-1)$ with $\sigma(\zeta_{2m})=\zeta^{-1}$. If $n=2$ then $A=\HQ(\Q)=\quat{-1,-1}{\Q}$, if $n=4$ then $A=\HQ(\Q(\sqrt{2}))$ and if $n=6$ then $A=\quat{-1,-3}{\Q(\sqrt{3})}$.  

In case b)iii) $G=Q_8\times M$ with $M$ a Z-group of order $m$ with $m$ and $o_2(m)$ odd. The latter means that $M$ is either cyclic or one of the groups of Case c) in \cite[Theorem~2.1.5]{ShirvaniWehrfritz1986}. However, the latter case must be excluded because then $G$ has an epimorphic image isomorphic to $K=Q_8\times C_{2^n}$ with $n\ge 4$ and $\Q K$ has a simple component isomorphic to $\HQ(\Q)\times \Q(\zeta_{2^n})\cong M_{\Q(\zeta_{2^n})}$. Therefore $M$ is cyclic. Again $G$ has a cyclic subgroup $H$ of index $2$ and $(H,1)$ is a strong Shoda pair which yields a division algebra isomorphic to $\HQ(\Q)\times \Q(\zeta_m)\cong \HQ(\Q(\zeta_m))$.
This finish the proof for Case 1.

\begin{table}[h!]
	$$	\matriz{{lll}
		\hline G & \text{Structure} &  \text{Epimorphic image} \\\hline\hline	
		\SG[24,3] &\SL(2,3) & \SG[12,3]=\PSL(2,3) \\ 
		\SG[40,3] & C_5\rtimes C_8 & \SG[20,3]=C_5\rtimes C_4 \\ 	
		\SG[48,28] &\mathcal{O^*}=2.S_4=\SL(2,3).2 & \SG[24,12]=S_4\\ 
		\SG[48,29] & \GL(2,3) & \SG[24,12]=\PSL(2,3) \\ 
		\SG[48,33] & C_4:\SL(2,3) &  \SG[12,3]=\PSL(2,3)\\ 
		\SG[64,37] & C_2 : (C_2^3 \rtimes C_4) & \SG[32,6]=C_3^2\rtimes C_4 \\
		\SG[72,19] & C_3^2\rtimes C_8 & \SG[32,6]=C_3^2\rtimes C_4 \\
		\SG[72,20] & (C_3\rtimes C_4)\times S_3 & \SG[36,10]=S_3\times S_3 \\
		\SG[72,22] & (C_6 \times S_3):2 & \SG[36,10]=S_3\times S_3 \\	
		\SG[72,24] & C_3^2:Q_8 & \SG[36,10]=S_3\times S_3 \\	
		\SG[72,25] & C_3 \times \SL(2,3) & \SG[12,3]=\PSL(2,3) \\	
		\SG[96,67] & \SL(2,3):C_4 &  \SG[24,12]=S_4 \\
		\SG[96,190] & (C_2\times \SL(2,3))\rtimes C_2 & \SG[24,12]=S_4  \\
		\SG[96,191] & (C_2:S_4)\rtimes C_2 & \SG[24,12]=S_4 \\
		\SG[96,202] & ((C_2 \times Q_8) \rtimes C2) \rtimes C3 & \SG[12,3]=A_4 \\
		\SG [120,5] & \SL(2,5) & \SG[60,5]=\PSL(2,5) \\
		\SG [128,937] & (Q_8\times Q_8)\rtimes C_2 & \SG[64,138]=(C_2^4\rtimes C_2)\rtimes C_2 \\	
		\SG[144,124] & C_3\rtimes (C_2:S_4) & \SG[24,12]=S_4\\
		\SG[144,128] & S_3\times \SL(2,3) & \SG[12,3]=\PSL(2,3)\\	
		\SG[144,135] & (C_3^2\rtimes C_8)\rtimes C_2 & \SG[36,9]=C_3^2\rtimes C_4 \\		
		\SG[144,148] & C_3^2 \rtimes ((C_4 \times C_2) \rtimes C_2) & \SG[36,10]=S_3^2 \\			
		\SG[160,199] & ((C_2\times Q_8)\rtimes C_2):C_5  & \SG[80,49]=C_2^4\rtimes C_5\\
		\SG[192,989] & (\SL(2,3) \rtimes C_4) \rtimes C_2  & \SG[24,12]=S_4 \\
		\SG[240,89] & C_2\rtimes S_5 &  \SG[120,34]=S_5 \\	
		\SG[240,90] & \SL(2,5)\rtimes C_2 &  \SG[120,34]=S_5 \\
		\SG[288,389] & C_3^2 \rtimes (C_4^2 \rtimes C_2)  & \SG[72,40]=S_3^2\rtimes C_2 \\
		\SG[384,618] & (Q_8^2\rtimes C_3)\rtimes C_2  & \SG[12,3]=A_4 \\
		\SG[384,18130] & (Q_8^2\rtimes C_3)\rtimes C_2  & \SG[24,12]=S_4\\
		\SG[720,409] & \SL(2,9) &   \SG[360,118]=\PSL(2,9) \\
		\SG[1152,155468] & ((Q_8^2 \rtimes C_3) \rtimes C_2) \rtimes C_3 &  \SG[288,1025]=((C_2^4 \rtimes C_3) \rtimes C_2) \rtimes C_3 \\
		\SG[1920,241003] & C_2:(C_2^4\rtimes A_5) & \SG[60,5]=A_5\\
		\hline
	}$$
	\caption{\label{NonAllowedQuotients}: Groups from \cite[Table~2]{EiseleKieferVanGelder2015} which have a proper epimorphic image in Table~\ref{ExcludedGroups}. The structure displayed is chosen to help in the recognition of the groups in the third column as epimorphic image of those in first column}
\end{table}

\textbf{Case 2}. Suppose that $A$ is not a division algebra. Then $A$ is exceptional of type (2), i.e. $A=M_2(D)$ with $D$ either $\Q$, an imaginary quadratic extension of $\Q$ or a totally definite quaternion algebra over $\Q$. Actually, by the main result in \cite{EiseleKieferVanGelder2015}, $D$ is either $\Q$, $\Q(\sqrt{-1})$, $\Q(\sqrt{-2})$, $\Q(\sqrt{-3})$, $H_1=\HQ(\Q)$, $H_3=\quat{-1,-3}{\Q}$ or $H_5=\quat{-2,-5}{\Q}$ and $G$ is one the groups in \cite[Table~2]{EiseleKieferVanGelder2015}. 
However $G$ is not any of the groups in the first column of \Cref{ExcludedGroups} or 
\Cref{NonAllowedQuotients}, because they have an epimorphic image in \Cref{ExcludedGroups}. 
This can be verified easily using GAP. For example, the following GAP calculation shows that $\SG[64,138]$ is an epimorphic image of $\SG[128,937]$:
\begin{verbatim}
G:=SmallGroup(128,937);;                       
NS:=NormalSubgroups(G);;
SSortedList(NS,x->IdSmallGroup(G/x));
[ [ 1, 1 ], [ 2, 1 ], [ 4, 2 ], [ 8, 3 ], [ 8, 5 ], [ 16, 11 ], [ 32, 27 ], 
[ 64, 138 ], [ 128, 937 ] ]
\end{verbatim}

Therefore $G$ is one of the groups in \cite[Table~2]{EiseleKieferVanGelder2015} not appearing neither in \Cref{ExcludedGroups} nor in \Cref{NonAllowedQuotients}. They are precisely the groups displayed in the last six rows of \Cref{EC}. Then, \cite[Theorem~3.7]{EiseleKieferVanGelder2015} implies that in each case $A$ is as in the corresponding line of the table.
\end{proof}

We are ready for the proof of Theorem \ref{Units}.

\begin{proofof}\emph{Theorem \ref{Units}}.
	Clearly (1) implies (2) and (3) implies (4) holds. The equivalence of (4) and (5) is the subject of  \cite[Corollary 3.2.8]{RibesZalesskii2010}. 
	The equivalence of (2) and (3) is proved in \cite[Proposition~2.2]{Jespers2021}.

	Let $e_1,\dots,e_n$ be the primitive central idempotents of $\Q G$.
Then $g\mapsto (ge_1,\dots,ge_n)$ defines an injective group homomorphism $f:G\rightarrow \prod_{i=1}^n Ge_i$. Each $Ge_i$ satisfies the hypothesis of the theorem because they are epimorphic images of $G$ and hence the Weddeburn components of $\Q Ge_i$ are also Wedderburn components of $G$. Moreover, each $Ge_i$ embedded in $\Q Ge_i$ which is either a division algebra or exceptional. By \Cref{ExceptionalComponents}, $Ge_i$ has an abelian normal subgroup $B_i$ such that $Ge_i/B_i$ has exponent dividing $4$. 
Then $G=\prod_{i=1}^n A_i$ is an abelian normal subgroup of $\prod_{i=1}^n Ge_i$ such that $(\prod_{i=1}^n Ge_i)/B$ has exponent dividing $4$. 
As $f$ is injective, $A=f^{-1}(B)$ is an abelian normal subgroup of $G$ with $G/A$ of exponent dividing $4$. In particular, $G$ is metabelian and by \cite[Theorem~31.1]{Sehgal1993}, $G$ has a torsion-free normal complement in $\U(\Z G)$, i.e. $\U(\Z G)=N\rtimes G$ with $N$ torsion-free. 

As $G$ is metabelian, (2) implies (1) is a consequence of a theorem of Withcomb \cite{Whitcomb1968} (see also \cite[Corollary~14.2.6]{Passman1977}). 	

	(4) implies (3) 
	By Theorem \ref{good}, $\widehat N$ is torsion-free. Consider the following commutative diagram
	
	$$\xymatrix{\U(\Z G)\ar[rd]\ar[d]&\\
		\widehat{\U( \Z G)}=\widehat{\U( \Z H)}\ar[r]&\U( \Z G)/N\\
		\U( \Z H)\ar[ru]\ar[u]&}$$
	
	Since $\widehat N$ is torsion-free, the restriction of the lower map to $H$ is injective, and so $H$ is isomorphic to a subgroup of $G$. In particular, $H$  satisfies the hypotheses of the theorem and so  has a torsion-free  complement $M$ whose profinite completion $\widehat M$ is torsion-free. Therefore, by symmetry $G$ is isomorphic to a subgroup of $H$. This implies that $G$ and $H$ are isomorphic.
\end{proofof}

\begin{rem}\label{ExceptionalConverse}
	The converse of Lemma~\ref{ExceptionalComponents} holds, i.e. every pair $A$ and $G$ appearing in the lemma satisfy the hypothesis of the lemma. 
	\Cref{WCLemma} displays the non-commutative Wedderburn components of the group algebras for all the groups in the lemma. 
	This can be verified for most groups using the GAP package \texttt{Wedderga}. For example one can compute the Wedderburn components of $\Q G$ for $G=\SG[72,37]$ with the following calculation:
	\begin{verbatim}
	G:=SmallGroup(72,30);;   
	gap> QG:=GroupRing(Rationals,G);;
	SSortedList(WedderburnDecompositionInfo(QG)$
	[ [ 1, Rationals ], [ 1, CF(3) ], [ 2, Rationals ], [ 2, CF(3) ] ]
	\end{verbatim} 
	which shows that the only non-commutative components are $M_2(\Q)$ and $M_2(\Q(\sqrt{-3}))$.

	Of course we cannot rely on GAP for the infinite families appearing in lines 5 and 6 of \Cref{EC}. For the second one can use that $\Q Q_8=4\Q \oplus \HQ(\Q)$ and $\Q C_m=\oplus_{d\mid m} \Q(\zeta_d)$. Then every non-commutative simple component of $\Q (Q_8\times C_m)$ is isomorphic to $\HQ(\Q)\otimes_{\Q} \Q(\zeta_d)=\HQ(\Q(\zeta_d))$. Since $d$ divides $m$ and $m$ and $o_m(2)$ are odd, so are $d$ and $o_d(2)$. This implies that $-1$ is not a sum of two squares in $\Q(\zeta_d)$ \cite{Moser1973} and hence $\HQ(\Q(\zeta_d))$ is a division algebra (see e.g. \cite[Proposition~1.6]{Pierce1982}). 
	
	Finally, if $G=C_3\rtimes C_{2^n}$ as in (2.e) then 
	$$\Q G = \Q C_{2^n} \oplus \oplus_{i=0}^{n-1} A_k$$
	with $A_k=\quat{\zeta_{2^k},-3}{\Q(\zeta_{2^k})}$. To prove this we can use the methods introduced in \cite{OlivieridelRioSimon2004} to compute the Wedderburn decomposition of group algebras of so called strongly monomial groups (see \cite[Section~3.5]{JespersdelRio2016}). Using \cite[Theorems~3.5.5 and 3.5.12 and Problem~3.4.3]{JespersdelRio2016} it follows that there is a one-to-one correspondence between the non-commutative Wedderburn components of $\Q G$ and the subgroups   of $Z(G)$. Observe that $Z(G)$ is cyclic of order $2^{n-1}$. The Wedderburn component associated to the subgroup $H_k$ of index $2^k$ in $Z(G)$ is the cyclic algebra $(\Q(\zeta_{3\cdot 2^k})/\Q(\zeta_{2^k}),\zeta_{2^k})$ which is the quaternion algebra described above. Moreover $A_0=M_2(\Q)$, $A_1=\quat{-1,-3}{\Q}$, $A_2=M_2(\Q(i))$ and $A_k$ is a division algebra for $k\ge 3$. 
	To prove that $A_2=M_2(\Q(i))$, observe that $i(1+i)^2-3=-1$. To prove that $A_k$ is a division algebra observe that if $m=3\cdot 2^k$ and $r$ is an integer satisfying $r\equiv 1 \mod 3$ and $r\equiv 1 \mod 2^k$ then in the notation of \cite{Amitsur1955}, $G/H_k=\U_{3\cdot 2^{k},r}$ and $A_k=\mathfrak{U}_{3\cdot 2^k,r}$. As $\U_{3\cdot 2^{k},r}$  is a subgroup of a division algebra, by \cite[Theorem~3]{Amitsur1955}, $A_k$ is a division algebra. 
	
	\begin{table}[h]
		$$\matriz{{llll}
			\hline G & \text{TDQA} & \text{Excep. 1} & \text{Excep. 2} \\\hline\hline			
			Q_8 & \HQ(\Q) \\\hline
			Q_{12} & \quat{-1,-3}{\Q} & & M_2(\Q) \\\hline
			Q_{16} & \HQ(\Q(\sqrt{2})) & & M_2(\Q) \\\hline
			Q_{24} & \HQ(\Q), \HQ(\Q(\sqrt{3})) & & M_2(\Q)  \\\hline
			C_3\rtimes C_{2^n}  & \quat{-1,-3}{\Q} & \quat{\zeta_{2^k},-3}{\Q(\zeta_{2^k})} (3\le k<n) & M_2(\Q),  \\
			(n\ge 4) & & & M_2(\Q(i)) \\\hline
			Q_8\times C_m & \HQ(\Q) & \HQ(\Q(\zeta_d))  (1\ne d) \\\hline
			S_3, D_8 & & & M_2(\Q) \\
			D_{12}, C_6\times S_3 \\\hline
			\SG[16,6]  & & & M_2(\Q(i)) \\
			\SG[16,13]\\\hline
			C_4\times S_3,  & & & M_2(\Q) \\
			\SG[32,11] & & &  M_2(\Q(i)) \\\hline
			\SG[16,8] & & & M_2(\Q) \\
			& & & M_2(\Q(\sqrt{-2})) \\\hline
			C_3\times S_3, C_3\times D_8 & & & M_2(\Q) \\
			\SG[24,8], \SG[72,30] & & &  M_2(\Q(\sqrt{-3})) \\\hline
			C_3 \times Q_8 & \HQ(\Q) & &  M_2(\Q(\sqrt{-3})) \\\hline
			\SG[36,6] & \quat{-1,-3}{\Q} & & M_2(\Q) \\
			& & & M_2(\Q(\sqrt{-3})) \\\hline
			\SG[32,8], \SG[32,44] &  & & M_2(\Q) \\
			\SG[64,137] & & & M_2(\HQ(\Q)) \\\hline
			\SG[32,50], &  & &  M_2(\HQ(\Q)) \\\hline
			&  & &  M_2(\Q) \\ 
			\SG[48,18]& \HQ(\Q(\sqrt{2})) & & M_2(\Q(\sqrt{-3})) \\
			& & & M_2\quat{-1,-3}{\Q} \\\hline		
			&  & &  M_2(\Q) \\ 
			\SG[48,39]&  & & M_2(\Q(i)) \\
			& & & M_2\quat{-1,-3}{\Q} \\\hline				
		}
		$$
		\caption{\label{WCLemma} The non-commutative simple components of the groups appearing in \Cref{ExceptionalComponents} classified by whether they are totally definite quaternion, exceptional of type 1 and exceptional of type 2.}
	\end{table}
\end{rem}

\begin{rem}
	Observe that the only exceptional components of type 1 appearing in \Cref{WCLemma} are $\quat{\zeta_{2^k},-3}{\Q(\zeta_{2^k})}$ with $k\ge 3$ and $\HQ(\Q(\zeta_d))$ with $1\le d$ and $d$ and $o_d(2)$ are odd. If $R$ is an order in one of these algebras then the real rank of $\SL_1(R)$ is greater than $1$. Indeed, for the center of both algebras are totally complex of degree $2^{k-1}$ and hence their rank is the number of places of their center which is $2^{k-2}$ for the first algebra and $\varphi(d)/2$ for the second one (see \cite[Examples~8.1.7]{Morris2015}). The first is greater than $1$ because $k\ge 3$. To see that $\phi(d)/2>1$ we use that $d>1$ and hence it is divisible by a odd prime $p$ and $o_p(2)$ is odd. This implies that $p\ge 7$ and hence $\varphi(d)\ge \varphi(p)\ge 6$. 
	
	This shows that the groups satisfying the hypothesis of \Cref{good} and \Cref{Units} are precisely those for which the groups of units of reduced norm 1 in an order of each Wedderburn component of the rational group algebra satisfies the Margulis Normal Subgroup Theorem (and  according to Serre's Conjecture should have the Congruence Subgroup Property, see the next section). 
\end{rem}

\section{Congruence Subgroup Problem}\label{CSP}

Our considerations in the previous sections  strongly connected to the Congruence Subgroup Problem. We state it here for the reader convenience.

Let $k$ be a global field and ${\bf G}$ be an almost simple,
connected, simply-connected algebraic group over $k$.
Let ${\bf G}(O)$ be the group of $S$-integral points in ${\bf G}$, where
${O}=O(S)$ is the {ring of $S$-integers in
$k$, for some non-empty, finite set $S$ of places $k$, containing
all the archimedean places.
An $S$-arithmetic group $\Gamma$ is a group commensurable with
${\bf G}(O)$.

Define the \emph{congruence} topology by taking the  subgroups
$$\Gamma(\alpha)=\{g\in {\bf G}(O)\cap \Gamma\}\mid g\equiv 1 ({\rm mod}\
\alpha)\}$$ corresponding to non-zero ideals $\alpha$ of $O(S)$ as
basis of neighbourhoods of the identity.  The completion
$\widetilde\Gamma$  of $\Gamma$ with respect to this topology is the
\emph{congruence completion}.

The \emph{congruence kernel}
$C=C(\Gamma)$ is the kernel of
the natural epimorphism
$\widehat{\Gamma}\longrightarrow \widetilde\Gamma.$

\begin{quote}
The  \emph{Congruence Subgroup Problem} (in modern understanding):  Compute  the congruence kernel $C$.
\end{quote}

The classical  congruence subgroup problem asked whether the congruence kernel is trivial; however all consequences of this hold also if the congurence kernel is finite. 
One says that $\Gamma$ has the Congruence Subgroup Property (CSP) if the congruence kernel is finite. 
J.-P. Serre \cite{serre} made the following

\begin{quote}
{\bf Conjecture.}   Let $S$ be  a  finite set  of  valuations  of $k$ that  contains  all  archimedean valuations if $k$ is a number field and is nonempty if $k$ has positive characteristic.  If the $S$-rank $rk_S{\bf G}:=\sum_{v\in S} rk_{k_v}{\bf G}\geq 2 $ (where $ rk_{k_v}({\bf G})$
is the dimension of a maximal $k_v$-split tori in ${\bf G}(k_v)$) and $rk_{k_v}{\bf G} >0$ for all non-archimedean $v\in S$ then the congruence kernel is finite (equivalently, central), i.e. CSP holds,  and if $rk_S{\bf G}=1$ then the congruence kernel is infinite, i.e. CSP does not hold.
\end{quote}

Note that arithmetic lattices satisfy the Margulis Normal Subgroup Theorem exactly when  $rk_S{\bf G}>1$ , i.e. when $\Gamma$ has  conjecturally  the Congruence Subgroup Property.

Suppose now that $\Gamma$ is the group of units with reduced norm 1 in an order of $A$. 
If $A$ is not exceptional then $\Gamma$ has the Congruence Subgroup Property. If $A$ is exceptional of type (2) then $\Gamma$ does not have the Congruence Subgroup Property . However if $A$ is exceptional of type (1) then  the Congruence Subgroup Property for $\Gamma$ is unknown. 

We can deduce now from Theorem \ref{good} the following statement about the congruence kernel of unit groups $U(\Z G)$ of a group ring $\Z G$ of a finite group whose rational algebra does not have exceptional components ot type (1). We  observe that the congruence kernel in this case is just the kernel of the natural homomorphism $\widehat{U(\Z G)}\longrightarrow U(\widehat\Z G)$. 

We denote by $vcd(\Gamma)$ the virtual cohomological dimension of the group $\Gamma$.

\begin{thm} Let $G$ be a finite group such that  $\Q G$ does not have exceptional  components of type (1) and let $C$ denote the congruence kernel of $\U(\Z G)$. Then  $vcd(C)\leq \sum_j (vcd(\Gamma_i))$, where $\Gamma_j$ runs via arithmetic lattices of the exceptional components of $\Q G$.
\end{thm}

\begin{proof} Let  $\Q G=\oplus_{i\in I} A_i$ be the decomposition into the simple components. Let $\widehat{U(\Z G)}$ be the profinite completion  of $U(\Z G)$. Choose an open subgroup  $H$  such that

(1) $H=\prod_i H_i$ with $H_i$ being the profinite completion of an arithmetic lattice of the component 	$A_i$, 

(2) $H_i$ intersects the (finite) congruence kernel $C_i$ of $\Gamma_i$  trivially for each not exceptional component $A_i$.

 As it was shown in the proof of Theorem \ref{good}, the group $\Gamma_j$  is good    for each $j$  such that $A_j$ is exceptional and hence so is $H_j\cap U(\Z G)$. Moreover, since vcd is stable for commensurability, $vcd(H_i)=vcd(\Gamma_i)$ for each $i$.  Thus $vcd(H_j)=vcd(\Gamma_j)$ for each $j$ such that $A_j$ is exceptional.   Since vcd of a closed subgroup does not exceed vcd of the  group,  $vcd(C_j)\leq vcd(H_j)$.  So the result follows from the fact that cohomological dimension of a direct product is the sum of cohomological dimensions of the factors.
\end{proof}

\end{document}